\documentclass[12pt]{elsarticle}
\usepackage{setspace}
\usepackage{amsmath,amscd, amsthm}
\usepackage{amsfonts}
\usepackage[english]{babel}
\usepackage{amssymb}
\usepackage[top=1in, bottom=1in, left=1in, right=1in]{geometry}
\usepackage{enumerate}
\usepackage{mathrsfs}

\usepackage{comment}
\usepackage{graphicx}
\usepackage{subcaption}

\usepackage[
  separate-uncertainty = true,
  multi-part-units = repeat
]{siunitx}
\usepackage{mathtools}
\def\multiset#1#2{\ensuremath{\left(\kern-.3em\left(\genfrac{}{}{0pt}{}{#1}{#2}\right)\kern-.3em\right)}}

\usepackage{tikz}
\usepackage{bbding} 
\usepackage{amsthm}
\usepackage{enumerate}
\usepackage{bm}
\usepackage{wasysym}
\usepackage[usestackEOL]{stackengine}[2013-10-15]
\stackMath

\usepackage{wrapfig}
\usepackage{subcaption}
\graphicspath{ {img/} }

\usepackage{pgfplots}

\newcommand{\where}{\;\ifnum\currentgrouptype=16 \middle\fi|\;}

\newcommand{\C}{\mathbb{C}}
\newcommand{\R}{\mathbb{R}}
\newcommand{\N}{\mathbb{N}}
\newcommand{\Z}{\mathbb{Z}}

\newtheorem{theorem}{Theorem}
\newtheorem{lemma}[theorem]{Lemma}
\newtheorem{proposition}[theorem]{Proposition}
\newtheorem{corollary}[theorem]{Corollary}
\theoremstyle{definition}

\theoremstyle{remark}

\theoremstyle{definition}

\title{Linear Operators, the Hurwitz Zeta Function and Dirichlet $L$-Functions }
\author{Bernardo Bianco Prado \& Kim Klinger-Logan}
\date{September 2019}

\begin{document}

\maketitle

\begin{quote}{\sc Abstract:} At the 1900 International Congress of Mathematicians, Hilbert claimed that the Riemann zeta function is not the solution of any algebraic ordinary differential equation its region of analyticity \cite{HilbertProb}. In 2015, Van Gorder addresses the question of whether the Riemann zeta function satisfies a {\it non}-algebraic differential equation and constructs a differential equation of infinite order which zeta satisfies \cite{RHequiv}.  However, as he notes in the paper, this representation is formal and Van Gorder does not attempt to claim a region or type of convergence. In this paper, we show that Van Gorder's operator applied to the zeta function does not converge pointwise at any point in the complex plane. We also investigate the accuracy of truncations of Van Gorder's operator applied to the zeta function and show that a similar operator applied to zeta and other $L$-functions does converge. Note that this version differs from the published version in Section 4.1 and 4.2.  
\\ \end{quote}

\section{Introduction}

In Hilbert's 1900 address at the International Congress of Mathematicians, he claimed that the Riemann zeta function
is not the solution of any algebraic ordinary differential equation
on its region of analyticity \cite{HilbertProb}. In \cite{VanGorder}, Van Gorder addresses the question of whether the Riemann zeta function satisfies a {\it non}-algebraic differential equation. As Van Gorder notes in the introduction of \cite{VanGorder}, it could be the case that $\zeta(z)$ satisfies a nonlinear differential equation or that it satisfies a linear differential equation of infinite order.\footnote{In fact, in \cite{GT}, Gauthier and Tarkhanov show that $\zeta(s)$ does satisify an inhomogeneous linear differential equation.  However, this equation is not algebraic.}  In \cite{VanGorder}, Van Gorder constructs a differential equation of infinite order that the Riemann zeta function satisfies \cite{RHequiv}.  However, as he notes in the paper, this representation is clearly formal and Van Gorder does not attempt to claim a region or type of convergence.\footnote{Though he does allude to some important things to be considered in Section 2 of \cite{VanGorder}.}

In what follows we will examine the region of convergence for the differential equation in question. We will also extend the formal identity appearing in Van Gorder's work to see that the Hurwitz zeta function satisfies a similar differential equation.  

 In Section \ref{T} we will begin with a brief overview of the differential operator introduced by Van Gorder.  In Sections \ref{titID} and \ref{hurDE} we will extend Van Gorder's main results to show that the Hurwitz zeta function formally satisfies a similar infinite order differential equation to the one in \cite{VanGorder}.  These results subsume those of Van Gorder.  In Section \ref{conv} we will address the issue of where such equations converge.  We will show that, in fact, the differential equation under investigation in \cite{VanGorder} diverges everywhere.  

We will see through the course of Section \ref{Hurwitz} that the formal arguments given by Van Gorder rely on a non-global characterization of his operator $T$ that only holds away from poles of the function on which it is being applied.  If we define a new operator $G$ in terms of this characterization globally, we can guarantee convergence. However, this new operator $G$ is not a {\it differential} operator and furthermore does not converge to Van Gorder's operator $T$. We will investigate this new operator $G$ in Section \ref{secG}.  

In Section \ref{secL} we will extend Van Gorder's argument to yield an operator equation involving $G$ applied to Dirichlet $L$-functions. We will see that the inverse $T^{-1}$ that Van Gorder presents in \cite{VanGorder} actually yields $G^{-1}$. In Section \ref{Bern} we make precise Van Gorder's claim that this inverse operator has a connection to the Bernoulli numbers, and in Section \ref{Ginvcon} we will use this connection to give identities for the Hurwitz zeta function and Dirichlet $L$-function and discuss the convergence of $G^{-1}$. 

In Section \ref{approx}, we examine the truncated version of the operator $T$.  Though $T$ does not converge when applied to $\zeta$, it is possible that some truncation of $T$ applied to $\zeta$ will provide a good approximation of Van Gorder's differential equation. 

\section{Van Gorder's operator applied to the Hurwitz Zeta Function}\label{Hurwitz}

\subsection{Van Gorder's Operator}\label{T}

The differential operator defined by VanGorder in \cite{VanGorder} is given by:
\begin{equation}
    T = \sum_{n=0}^\infty L_n
\end{equation}
where
\begin{align*}
    L_n &:= p_n(s) \exp(nD)\\
    p_n(s) &:= \begin{cases}
    1& \text{ if } n = 0\\
    \frac{1}{(n+1)!}\prod_{j = 0}^{n-1} (s+j) &\text{ if } n >0 \end{cases}\\
    \exp(nD) &:= id + \sum_{k=1}^\infty \frac{n^k}{k!} D_s^k
\end{align*}
for $D_s^k := \frac{\partial ^k}{\partial s^k}$. For an overview of infinite order differential equations see Charmichael's \cite{Charmichael} and for more recent applications involving infinite order differential equations with initial conditions see \cite{BarnabyK}.

Van Gorder notes that $\exp(nD)$ acts as a shift operator for meromorphic functions in the sense that $\exp(nD)u(s) = u(s+n)$ {\it sufficiently far away from poles}. However, he does not attempt to answer the question of precisely what is ``sufficiently far away from poles" but instead references Ritt's \cite{Ritt}.  As we will see in Section \ref{HurCon}, the operator that Ritt considers, $\exp(D)$, (though of infinite order)  is simpler than Van Gorder's $T=\sum_{n=0}^\infty p_n(s) \exp(nD)$.  Thus more work is necessary to address the convergence of $T$ than is done by Ritt \cite{Ritt}.

In \cite{VanGorder}, Van Gorder proves that \begin{equation}\label{zetaDE}T[\zeta(s)-1]=\frac{1}{s-1}\end{equation} formally.
The crux of the proof relies upon the characterization of $\exp(nD)$ as the ``shift operator".  In the following two sections, we prove that the Hurwitz zeta function satisfies a similar equation

   \begin{equation}\label{HurDE}
    T\left[\zeta (s,a) - \frac{1}{a^s}\right] = \frac{1}{(s-1)a^{s-1}}.\end{equation}
 Our argument is akin to that of Van Gorder's. 
 
 It is important to note that in the proof of Corollary \ref{Diffeq} we are assuming that $\exp(nD)u(s) = u(s+n)$ when claiming
\begin{equation}\label{L_n}
L_n \left[ \zeta (s, a) - \frac{1}{a^s} \right] = p_n(s)\left(\zeta (s+n, a) - \frac{1}{a^{s+n}}\right)
\end{equation}
This assumption is also made at a similar place in \cite{VanGorder}.
 However since this is only true ``sufficiently far away from the poles" of $u$, this leads to the natural question of {\it where} (\ref{zetaDE}) and  (\ref{HurDE}) hold.  We will begin to address this question by examining the convergence of the differential operator $T$ in Section \ref{conv}. 

\subsection{A useful identity for the Hurwitz Zeta Function}\label{titID}

In order to show that $\zeta$ formally solves the differential equation (\ref{zetaDE}), Van Gorder uses the following identity
\begin{equation}\label{zetaID1}\zeta(s) = \frac{s}{s-1} -\sum_{n=1}^{\infty}\frac{\prod_{j=0}^{n-1}(s+j)}{(n+1)!}(\zeta(s+n) -1)\end{equation}
 which can be found in \cite{Apostol} and \cite{Titchmarsh1986}.
 Following the argument of Titchmarsh \cite{Titchmarsh1986}, we need to generalize the identity to the Hurwitz zeta function.

\begin{lemma}\label{id1}
Let $\zeta (s, a)$ be the Hurwitz zeta function. Then, for $\text{Re}(s) > 2$ and $0 < a \leq 1$, we have that 
\begin{equation}\label{HurID1} \zeta (s, a) - \frac{1}{(s-1)a^{s-1}} = \frac{1}{a^{s}} - \sum_{n = 1}^{\infty} \frac{\prod_{j = 0}^{n-1} (s+j)}{(n+1)!}\left(\zeta (s+n,a) - \frac{1}{a^{s+n}} \right)\end{equation}
\end{lemma}

\begin{proof}
Let $s \in \C$ satisfy $\text{Re}(s) > 2$. We can then write the Hurwitz zeta as a series and it suffices to show that the series
\begin{equation}\label{flipseries}
    \sum_{k=1}^\infty\sum_{n = 1}^{\infty}   \frac{\prod_{j = 0}^{n-1} (s+j)}{(n+1)!}\left( \frac{1}{(k+a)^{s+n}} \right) 
\end{equation}
converges absolutely pointwise to $\frac{1}{(s-1)a^{s-1}} + \frac{1}{a^s} - \zeta(s,a)$, since this will mean that we can interchange the order of summation by Fubini's Theorem. To see why we get such result, observe that from geometric series we have that for an integer $k \geq 0$,
\begin{align}
    \left( \frac{k+a}{k+a-1} \right)^{s-1}  &= \left( \frac{1}{1-\frac{1}{k+a}} \right)^{s-1} \notag\\
    &= \sum_{n=0}^\infty \frac{\prod_{j = 0}^{n - 1}(s-1+j)}{n!}\frac{1}{(k+a)^n} \notag \\
    &= \sum_{n=0}^\infty \frac{\prod_{j = -1}^{n - 2}(s+j)}{n!}\frac{1}{(k+a)^n} \label{geopower}
\end{align}
where the last series is absolutely convergent since by the triangle inequality and  geometric series, we have
\begin{align*}
    \sum_{n=0}^\infty \left|\frac{\prod_{j = -1}^{n - 2}(s+j)}{n!}\frac{1}{(k+a)^n} \right| &\leq \sum_{n=0}^\infty \frac{\prod_{j = -1}^{n - 2}(|s|+j)}{n!}\frac{1}{(k+a)^n} = \left( \frac{k+a}{k+a-1} \right)^{|s|-1}
\end{align*}
By absolute convergence, $\frac{1}{(k+a)^{s-1}(s-1)}\sum_{n=2}^\infty \frac{\prod_{j = -1}^{n - 2}(s+j)}{n!}\frac{1}{(k+a)^n} $ is precisely the $k^{th}$ term of the left summation of (\ref{flipseries}). But this is the same as the expression $\frac{1}{(k+a)^{s-1}(s-1)}\left[ \left( \frac{k+a}{k+a-1} \right)^{s-1}  - 1-\frac{s-1}{k+a}\right]$ and we have that
\begin{align*}
     \frac{1}{(s-1)a^{s-1}} + \frac{1}{a^s} - \zeta(s,a) &= \frac{1}{(s-1)a^{s-1}} - \sum_{k=1}^\infty \frac{1}{(k+a)^s} \\
    &= \frac{1}{s-1}\sum_{k=0}^\infty \frac{1}{(k+a)^{s-1}} - \sum_{k=1}^\infty \frac{1}{(k+a)^s} - \frac{1}{s-1}\sum_{k=1}^\infty \frac{1}{(k+a)^{s-1}}
\end{align*}
Since these three series converge absolutely, we can re-index the leftmost series and get that this is equal to an absolutely convergent series given by
\begin{align*}
    &\sum_{k=1}^\infty \left( \frac{1}{(s-1)(k-1+a)^{s-1}} - \frac{1}{(s-1)(k+a)^{s-1}} - \frac{1}{(k+a)^s} \right)\\ &= \sum_{k=1}^\infty \frac{1}{s-1}\cdot \frac{1}{(k+a)^{s-1}}
    \left[ \left( \frac{k+a}{k+a-1} \right)^{s-1} -1 - \frac{s-1}{k+a} \right]
\end{align*}
Which is an absolutely convergent series and gives the desired result.
\end{proof}

In a more elegant way, we can express this identity in terms of the $\Gamma$ function using the fact that $\frac{\Gamma(s+n)}{\Gamma(s)} = \prod_{j = 0}^{n-1} (s+j)$ for $s \in \C, n\in \N$. Equation (\ref{id1}) then becomes:

\begin{equation}\label{Gamma}\zeta (s, a) - \frac{1}{(s-1)a^{s-1}} = \frac{1}{a^{s}} - \sum_{n = 1}^{\infty} \frac{\Gamma(s+n)}{(n+1)!\Gamma(s)}\left(\zeta (s+n,a) - \frac{1}{a^{s+n}} \right)\end{equation}

We now show that this identity holds for all $s\in \C$.

\begin{lemma}\label{absconv}
The right-hand side of equation (\ref{HurID1}) in Lemma 1 converges absolutely for all $s \in \C$.
\end{lemma}

\begin{proof}
We first need to treat a delicate point. That is, equation (\ref{HurID1}) is well-defined when $s + n = 1$ for some integer $n \geq 0$. Namely, for such $n$, it makes sense to have the $n^{th}$ term of the sum be $\frac{\prod_{j=0}^{n-1}(s+j)}{(n+1)!}\left(\zeta(s+n,a) - \frac{1}{a^{s+n}}\right)$. Since $(s+n-1)\zeta(s+n,a)$ cancels the pole of the Hurwitz $\zeta$ at $s+n$ and $s+n-1$ appears as the last term in $\prod_{j=0}^{n-1}(s+j)$, the series is well-defined.

To show convergence, let $s \in \C$ and let $N > 0$ be an integer so that Re$(s+N) > 1$. It suffices to show that $$ \sum_{n = N}^{\infty} \frac{\prod_{j = 0}^{n-1} (s+j)}{(n+1)!}\left(\zeta (s+n,a) - \frac{1}{a^{s+n}} \right)$$ converges absolutely. First, we bound $\left|\zeta (s+n,a) - \frac{1}{a^{s+n}}\right|$. Since, $n \geq N$, we have that $\text{Re}(s+n) \geq \text{Re}(s+N) > 1$. By the triangle inequality and by the integral inequality for non-negative series, we can write
\begin{align*}
    \left|\zeta (s+n,a) - \frac{1}{a^{s+n}}\right| &= \left| \sum_{k=1}^\infty \frac{1}{(k+a)^{s+n}} \right| \\
    &\leq \sum_{k=1}^\infty \left|\frac{1}{(k+a)^{s+n}}\right| \\
    &= \sum_{k=1}^\infty \frac{1}{(k+a)^{\sigma+n}} \\
    &\leq \frac{1}{(1+a)^{\sigma+n}}+  \int_{1}^\infty \frac{1}{(x+a)^{\sigma+n}} dx \\
    &= \frac{1}{(1+a)^{\sigma+n}} + \frac{1}{\sigma+n-1}\cdot \frac{1}{(1+a)^{\sigma +n}} \\
    &= \frac{\sigma+n}{\sigma+n-1}\cdot \frac{1}{(1+a)^{\sigma+n}} \leq \frac{\sigma+n}{\sigma+n-1}\cdot \frac{1}{2^{\sigma+n}}
\end{align*}
In addition, observe that $\left| \prod_{j = 0}^{n-1} (s+j)  \right| \leq  \prod_{j = 0}^{n-1} (|s|+j)$. We then have that 
\begin{align*}
    \sum_{n = N}^{\infty} \left|\frac{\prod_{j = 0}^{n-1} (s+j)}{(n+1)!}\left(\zeta (s+n,a) - \frac{1}{a^{s+n}} \right)\right| \leq \sum_{n = N}^{\infty} \frac{\prod_{j = 0}^{n-1} (|s|+j)}{(n+1)!}\cdot \frac{\sigma+n}{\sigma+n-1}\cdot \frac{1}{2^{\sigma+n}}
\end{align*}
Now since  $\sum_{n = N}^{\infty} \frac{\prod_{j = 0}^{n-1} (|s|+j)}{(n+1)!}\cdot \frac{1}{2^{|s| +n}}$ converges and 
\begin{align*}
    \lim_{n\to \infty} \frac{2^{|s|+n}(\sigma+n)}{2^{\sigma+n}(\sigma+n-1)} = 2^{|s|-\sigma}
\end{align*} by the Limit Comparison test, $\sum_{n = N}^{\infty} \frac{\prod_{j = 0}^{n-1} (|s|+j)}{(n+1)!}\cdot \frac{\sigma+n}{\sigma+n-1}\cdot \frac{1}{2^{\sigma+n}}$ must also converge.  Thus, $ \sum_{n = N}^{\infty} \frac{\prod_{j = 0}^{n-1} (s+j)}{(n+1)!}\left(\zeta (s+n,a) - \frac{1}{a^{s+n}} \right)$ converges absolutely.

\end{proof}

Since our series converges absolutely for all $s \in \C$, we must have that our identity in Lemma 1 actually holds for all $s \in \C\setminus\{1\}$. Thus, we have the following corollary.

\begin{corollary}\label{cor}
For all $s \in \C\setminus\{1\}$, we have the following identity $$\zeta (s, a) = \frac{1}{(s-1)a^{s-1}} + \frac{1}{a^{s}} - \sum_{n = 1}^{\infty} \frac{\prod_{j = 0}^{n-1} (s+j)}{(n+1)!}\left(\zeta (s+n,a) - \frac{1}{a^{s+n}} \right)$$
\end{corollary}
\qed

Following Van Gorder's use of equation (\ref{zetaID1}) in \cite{VanGorder}, we will use equation (\ref{HurID1}) to show that (\ref{HurDE}) holds formally.

\subsection{The Hurwitz zeta function formally satisfies a differential equation}\label{hurDE}



    





Now we will show that the Hurwitz zeta function formally satisfies the differential equation (\ref{HurDE}).  This result is a generalization of Theorem 3.1 from \cite{VanGorder}. 
%

\vspace{.5cm}

\begin{corollary}\label{Diffeq}
    Let $T$ be as defined above. Then $\zeta (s, a)$ formally satisfies the differential equation 
$$
    T\left[\zeta (s,a) - \frac{1}{a^s}\right] = \frac{1}{(s-1)a^{s-1}}
$$
for $s \in \C$ satisfying $s + n \neq 1$ for all $n \in \Z_{\geq 0}$.
\end{corollary}

\begin{proof}
Using equation (\ref{L_n})
$$L_n \left[ \zeta (s, a) - \frac{1}{a^s} \right] = p_n(s)\left(\zeta (s+n, a) - \frac{1}{a^{s+n}}\right)
$$
we have
\begin{align*}
    T\left[ \zeta (s, a) - \frac{1}{a^s} \right] &= \sum_{n=0}^\infty L_n \left[ \zeta (s, a) - \frac{1}{a^s} \right]\\
    &= \sum_{n=0}^\infty p_n(s)\left(\zeta (s+n, a) - \frac{1}{a^{s+n}}\right)\\
    &= \zeta (s, a) - \frac{1}{a^{s}} + \sum_{n=1}^\infty \frac{\prod_{j = 0}^{n-1} (s+j)}{(n+1)!} \left(\zeta (s+n, a) - \frac{1}{a^{s+n}}\right)\\
    &= \frac{1}{(s-1)a^{s-1}}
\end{align*}
\end{proof}


\subsection{Convergence}\label{conv}
In this section we will show that $T$ applied to the Hurwitz zeta function does not converge. However, for certain analytic functions $f$, we see in Section \ref{p_n} that $Tf$ does converge.

\subsubsection{Convergence of $T$ when applied to the Hurwitz Zeta-Function}\label{HurCon}
As Van Gorder notes on page 781 of \cite{VanGorder}, ``we must exercise some caution when working with infinite
order differential equations if we are concerned with convergence of the operators
near poles of the functions being operated upon." As a basis for this concern the author alludes to \cite{Ritt} where Ritt establishes formally that $(\exp(D) - z)\Gamma(z) = 0$.  Ritt notes that $\Gamma$ does not satisfy this differential equation on all of $\C$ but away from the infinitely many poles of $\Gamma$. Of course, in the case of the Riemann zeta function and Hurwitz zeta function, we only have one pole to be concerned about.  As Van Gorder states, the operator $\exp(D)$ is only valid ``outside of a neighborhood of the pole at $z=1$." Our goal is to investigate {\it which} neighborhood and examine its effect on the convergence of the operator $T$ as it is applied to zeta functions.  In what follows we will consider $T$ applied to the Hurwitz zeta function since (when $a=1$) it also covers the case of the Riemann zeta function.

 Recall that in the proof of Corollary \ref{Diffeq}, our use of equation (\ref{L_n}) relies upon the characterization of $\exp$ as the shift operator $\exp(nD)[\zeta(s,a)] = \zeta(s+n,a)$ for $s \in \C$ and $0<a \leq 1$ away from the poles of $\zeta$. This characterization comes from the Taylor series expansion for $\zeta$; thus, we must consider the radius of convergence of the Taylor series when considering when this characterization holds.
 
   The critical observation is that this operator is, \textit{formally}, the Taylor series about a point $s \in \C$ evaluated at $z\in \C$. Namely, the formal Taylor series is
\begin{align}\label{Taylor}
    \zeta(z,a)- \frac{1}{a^z} = \sum_{k = 0}^\infty \frac{D_s^k \left( \zeta(s,a) - \frac{1}{a^s}\right)}{k!}(z-s)^k
\end{align} which, at the point $s+n$ for $n \in \Z_{\geq 0}$, will {formally} satisfy
\begin{align}
    \zeta(s+n,a) - \frac{1}{a^{s+n}} = \sum_{k = 0}^\infty \frac{D_s^k \left( \zeta(s,a) - \frac{1}{a^s}\right)}{k!}n^k = \exp(nD)[\zeta(s,a)]
\end{align}
Since $\zeta(z,a)$ has a pole at $z=1$, these series converge pointwise for $|z-s|<|s-1|$ and $|(s+n)-s|=|n|<|s-1|$.


    We claim that this operator applied to the function $\zeta(s,a)- \frac{1}{a^s}$ does not converge pointwise {\it anywhere}. More explicitly for any $s \in \C$, the sequence of partial sums of the series $T\left[\zeta(s,a) - \frac{1}{a^s}\right]$ is not a well-defined sequence of complex numbers. This comes from the fact that, to have a well-defined a complex-valued series, we need, first, a sequence of complex numbers $(z_n)_{n= 0}^\infty$ so we can define the sequence of partial sums $S_N=\sum_{n =0}^N z_n$ which is, again, a sequence of complex numbers. Then, if the sequence of partial sums converges to a complex number $S$, we write $\sum_{n =0}^{\infty} z_n = S$. What we will show now is that, for any $s \in \C$, the definition of the operator $T$ evaluated at $\zeta(s,a)- \frac{1}{a^s}$ fails this first step by failing to make a sequence of complex numbers.

\begin{proposition}\label{taylordiv}
For any $s \in \C$, we can find some $N \geq 0$ so that the series $$\exp(ND)\left[\zeta(s,a) - \frac{1}{a^s}\right]=\sum_{k = 0}^\infty \frac{D_s^k \left( \zeta(s,a) - \frac{1}{a^s}\right)}{k!}N^k $$ diverges.
\end{proposition}

\begin{proof}
Let $s \in \C$. If $s = 1$, then, the term at $k = 0$ of series (\ref{Taylor}) evaluated at $z = 1$ is undefined, so the series is not well-defined. In this case, $N = 0$ satisfies our claim.

To complete our proof, let $s \neq 1$. By Taylor's Theorem, there is a radius of convergence $r \geq 0$ so that the series (\ref{Taylor}) converges absolutely when evaluated at $z \in \C$ satisfying $|z-s| < r$. In addition, it must diverge when evaluated at $z \in \C$ satisfying $|z-s| > r$. Now, since $\zeta(z,a)$ has a pole at $z = 1$, we have that series (\ref{Taylor}) cannot converge when evaluated $z = 1$. 
Thus, we must have that $|s-1| \geq r$. 

Let $N > 0$ satisfy $|s+N - s| = N > |s-1|>r$. Then, we have that the series (35) evaluated at $z = s+N$ 
 $$\zeta(s+N,a)- \frac{1}{a^s} = \sum_{k = 0}^\infty \frac{D_s^k \left( \zeta(s,a) - \frac{1}{a^s}\right)}{k!}N^k=\exp(ND)\left[\zeta(s,a) - \frac{1}{a^s}\right]$$
 must be divergent. So we have found the desired $N \geq 0$.
\end{proof}

\begin{lemma}\label{p_nne0} For $s \in \Z_{\geq 0}$ and $n>0$, $p_n(s)\neq 0$. Specifically, $p_n(s)\geq   \frac{1}{(s-1)!(n+1)}$ for $s \in \Z_{> 0}$ and $n>0$.
\end{lemma}

\begin{proof}

For $s=0$, notice that for $n>1$, $p_n(0)=\frac{(n-1)!}{(n+1)!}= \frac{1}{n(n+1)}\neq 0$. Now, let $s \in \Z_{> 0}$. Observe that, for all $n > 0$, we have $p_n(s) = \frac{(s+n-1)!}{(s-1)!(n+1)!} \geq \frac{(1+n-1)!}{(s-1)!(n+1)!} = \frac{1}{(s-1)!(n+1)}\neq 0$.
\end{proof}

\begin{lemma}\label{eq0} For $s\in \Z_{< 0}$ and $n\geq1-s$, $p_n(s)=0$.

\end{lemma}
\begin{proof} 
Let $s \in \Z_{< 0}$. For $n>1$, $p_n(s)=\frac{1}{(n+1)!}\prod_{j=0}^{n-1}s+j$. Note that $1-s\in\Z_{\geq 0}$ and so for $N:=1-s$, \begin{align*}\prod_{j=0}^{N-1}s+j &=(s+N-1)(s+N-2)\dots (s+2)(s+1) \cdot s \\ &=(s+1-s-1)(s+1-s-2)\dots (s+2)(s+1) \cdot s = 0 \end{align*} Thus for each $n\geq N$, $p_n(s)=0$.
\end{proof}

\begin{theorem}\label{main} $T\left[ \zeta (s, a) - \frac{1}{a^s} \right]=\sum_{n=0}^{\infty} p_n(s)\exp(nD)\left[\zeta(s,a)- \frac{1}{a^s}\right]$ diverges for all complex numbers $s\in \C$.\end{theorem}
\begin{proof} From Proposition \ref{taylordiv},  since for all $s\in \C$ and all $n \geq 0$, we can find $N \geq 0$ so that $p_N(s)\exp(ND)[\zeta(s,a)- \frac{1}{a^s}]$ is divergent whenever $p_n(s)\neq 0$. Since $p_n$ can be defined in terms of the $\Gamma$-function as in equation (\ref{Gamma}), $p_n(s)$ can only be equal to zero if $s \in \Z_{< 0}$. 

Assume that $s \in \Z_{< 0}$.  By Lemma \ref{eq0}, for $n\geq 1-s$, $p_n(s)=0$.  By Proposition \ref{taylordiv}, there is some $N\geq 0$ so that $\exp(ND)\left[\zeta(s,a) - \frac{1}{a^s}\right]$ diverges.  If $N<1-s$, then $p_N(s)\neq 0$ and $p_N(s)\exp(ND)\left[\zeta(s,a) - \frac{1}{a^s}\right]$ diverges as above.

If $s \in \Z_{< 0}$ and $N\geq1-s$, then $p_N(s) = 0$.  However, $0\cdot \exp(ND)\left[\zeta(s,a)- \frac{1}{a^s}\right]$ is also not a complex number since $\exp(ND)\left[\zeta(s,a)- \frac{1}{a^s}\right]$ diverges.

Recall that $T=\sum_{n=0}^{\infty}p_n(s) \exp(nD)$. Then, for any $s \in \C$, we can find some $N \geq 0$ so that the $N^{th}$ partial sum of $T[\zeta(s,a)- \frac{1}{a^s}]$, $$\sum_{n=0}^N p_n(s)\exp(nD)\left[\zeta(s,a)- \frac{1}{a^s}\right]$$
is not a complex number. 
 Thus, we cannot define the series $T[\zeta(s,a)- \frac{1}{a^s}]$ at any such points $s$. We conclude the series does not converge in $\C $.
\end{proof}




\subsubsection{Convergence of $T$ in a General Setting}\label{p_n}

We now wish to discuss the convergence of $T$ in a more general setting. To do so, we first look at $T$ applied to the constant function with the goal of understanding the behavior of the series $\sum_{n=0}^\infty p_n(s)$ for $s \in \C$. 

\

\begin{lemma}\label{c1}
For $s \in \Z_{> 0}$, the series $\sum_{n =0}^\infty p_n(s)$ diverges, and for $s\in \Z_{\leq 0}$, the series $\sum_{n =0}^\infty p_n(s)$ converges.
\end{lemma}

\begin{proof}
Let $s \in \Z_{> 0}$. From Lemma \ref{p_nne0}, for all $n > 0$, we have $p_n(s) \geq \frac{1}{(s-1)!(n+1)}$. Then, by series comparison, we have that, since $ \sum_{n=1}^\infty \frac{1}{n}$ diverges, then $\sum_{n=1}^\infty \frac{1}{(s-1)!(n+1)}$ also diverges. By comparison with $\sum_{n=1}^\infty \frac{1}{(s-1)!(n+1)}$, the series  $\sum_{n=1}^\infty p_n(s)$ diverges.

For $s=0$, notice that for $n>1$, $p_n(0)=\frac{(n-1)!}{(n+1)!}= \frac{1}{n(n+1)}$ and so $\sum_{n =0}^\infty p_n(0)$ converges. By the proof of Lemma \ref{eq0}, for $s \in \Z_{< 0}$ and $N=1-s$, we have $\prod_{j=0}^{N-1}s+j = 0.$  Thus for each $n\geq N$, $p_n(s)=0$ and so the series $\sum_{n =0}^\infty p_n(s)$ converges for $s \in \Z_{< 0}$.
\end{proof}

It is more difficult to determine what happens outside of $ \Z$.


\begin{corollary}\label{constant}
For $s \in \R$ with $s > 1$, the series $\sum_{n =0}^\infty p_n(s)$ diverges.
\end{corollary}

\begin{proof}
Let $s \in \R$ satisfy $s \geq 1$. First, observe that, for all integers $j \geq 0$, we also have that $s+j\geq 1+j$. This means that $p_n(s) \geq p_n(1)$ and since, by Lemma \ref{c1}, $\sum_{n=1}^\infty p_n(1)$ diverges we have that, by series comparison, $\sum_{n=1}^\infty p_n(s)$ also diverges.
\end{proof}

\begin{corollary}\label{constant}
For $s \in \C$ with $\text{Re}(s) > 1$, the series $\sum_{n =0}^\infty p_n(s)$ does not converge absolutely.
\end{corollary}

\begin{proof}
Let $s \in \C$ satisfy $\text{Re}(s) \geq 1$. For all integers $j \geq 0$,  note that $|s+j| \geq 1+j$. This means that $|p_n(s)| \geq |p_n(1)|$ and since, by Lemma \ref{c1}, $\sum_{n=1}^\infty |p_n(1)|$ diverges we have that, by series comparison, $\sum_{n=1}^\infty |p_n(s)|$ also diverges.
\end{proof}







We can guarantee convergence when restricting to $L^1(\R_{\geq 0})$. First, we need to generalize $p_n(s)$. Writing it as $p(n,s) \coloneqq p_n(s)$ and observing that, for all integers $n\geq 0$, we have that $p(n,s)= \frac{\Gamma(s+n)}{\Gamma(n+1)\Gamma(s)}$, we can make sense of $p(n,s)$ when $n$ is any real number. We first define the set $U = \{s \in \C: s\neq n \text{ for all } n \in \Z_{< 0}\}$. We, then, consider $p : \R_{\geq 0} \times U \to \C$ given by $p(x,s) = \frac{\Gamma(s+x)}{\Gamma(x+1)\Gamma(s)}$, which makes sense for all $x \in \R_{\geq 0}$ and all $s \in U$.

\begin{proposition}\label{l1conv}
If $f:\C \to \C$ is analytic with radius of convergence equal to $+\infty$ and if we have that, for all $s \in U$, the function $p(\cdot,s)f(\cdot + s)$ is in $L^1(\R_{\geq 0})$, then $T[f](s)$ converges absolutely for all $s \in U$.
\end{proposition}
\begin{proof}
Let $s \in U$. By assumption, $\int_0^\infty |p(x,s)f(x+s)|dx<\infty$. Then, by the integral test for series, $\sum_{n=0}^\infty |p_n(s)f(s+n)|$ must converge, as desired. 
\end{proof}

This means that, in such cases, we can define a function $g : U \to \C$ given by $g(s) = T[f](s)$. Our next result seeks to give a sufficient condition for uniform convergence.

\begin{proposition}\label{uc1}
If $f : \C \to \C$ is analytic on all of $\C$ with radius of convergence equal to $+\infty$ and if there are constants $a,b \in \R\cup\{-\infty\}$ and $c \in \R \cup\{+\infty\}$ with $b <c$ so that, when we consider the set $U := \{z \in \C: \text{Re}\,(z) \in (a,\infty) \text{ and } \text{Im}(z) \in (b,c)\}$, we have that $\sum_{n = 0}^\infty ||p_n(s) f(s+n)||_{C^\infty(U)}$ converges, then $T[f]$ converges uniformly to a continuous function $g : U \to \C$. 
\end{proposition}

\begin{proof}
Let $s \in \C$. Since $f$ is analytic with radius of convergence equal to $+\infty$, for all integers $n>0$, we have that $\sum_{k=0}^\infty \frac{f'(s)}{k!}n^k$ converges to $f(s+n)$ by Taylor's theorem. Then, we have that for each $s \in U$,
\begin{align*}
\left|p_n(s) \exp(nD)[f(s)]\right|=|p_n(s)f(n+s)|\leq ||p_n(s)f(n+s)||_{C^\infty(U)}
\end{align*}
and
\begin{align*}
|T[f](s)| &=\left|\sum_{n=0}^{\infty} p_n(s) \exp(nD)[f](s)\right| \\ 
&\leq  \sum_{n = 0}^\infty \left|\left|p_n(s)\left( \limsup_{K \to \infty}\sum_{k=0}^K \frac{f'(s)}{k!}n^k \right)\right|\right|_{C^\infty(U)}
= 
\sum_{n=0}^\infty ||p_n(s) f(s+n)||_{C^\infty(U)}<\infty
\end{align*}

By the Weierstrass M-test, we have that $\sum_{n = 0}^\infty (p_n(\cdot) \exp(nD)[f(\cdot)])\big|_{U}$ converges uniformly to a continuous function $g : U \to \C$.
\end{proof}





We can have an analogous result when $U = \{z \in \C : \text{Re}(z) \in [a,\infty)$ and $\text{Im}(z) \in [c,b]\}$ as well as when $\text{Im}(z)$ is in a half-open interval.

\section{Generalizing Van Gorder's Operator}\label{secG}

The main reason the operator $T$ is not well-defined when applied to $\zeta$ is because $\zeta$ is not analytic and so the radius of convergence of its Taylor series expanstion is not $+\infty$. Specifically, the problem is that for all $s \in \C$, we can find some $N>0$ for which $\exp(nD)[\zeta (s,a)]$ will not be convergent. However, when treating the operator $\exp(nD)$ as  the shift operator, formally, we are able to show that (\ref{zetaDE}) and (\ref{HurDE}) hold.
With that in mind, we define an operator $G$, which agrees with $T$ on analytic functions with radius of convergence equal to $+\infty$ but which can be applied to a wider range of functions. 

Let $\mathcal{M}$ be the collection of meromorphic functions on $\C$ and $f\in\mathcal{M}$. Define $G:\mathcal{M}\to \mathcal{M} $ by 
\begin{align}\label{G}
G[f](s) = \sum_{n=0}^\infty p_n(s)f(s+n)
\end{align}

For this operator to be well-defined, we do not  require that $f$ be differentiable, a significant gain from the definition of $T$. Note that $G$ agrees with $T$ on analytic functions. Thus $G$ satisfies a version of Proposition \ref{l1conv} and Proposition \ref{uc1}. The assumption that $f$ be analytic may be weakened in such versions.

When $G$ is applied to $\zeta(s,a)$, we recover the identity (\ref{HurID1}) in Lemma \ref{id1} and we conclude that $G[\zeta(\cdot,a)]$ converges pointwise to a continuous function defined on $\C\setminus \{1\}$. 

\subsection{Using $G$ to get an identity for the $\zeta$-function}

Recalling our discussion of $T$, observe that when we evaluate $\exp(nD)\left[\zeta(-m,a)-\frac{1}{a^{-m}}\right]$ for $m>0$ and $n\geq 0$ integers, we have convergence of the Taylor series whenever $n \leq m$ because $|-m+n-(-m)| = n < |-m-1| = m+1$ and because $m+1$ is the radius of convergence of such series as we discussed in Section \ref{HurCon}. In addition, from the proof of Lemma \ref{eq0}, $p_n(-m) = 0$ for $n > m$. In addition, by the fact that the pole of $\zeta(s,a)$ at $s=1$ has residue $1$, we have that
\[
\lim_{s\to -m} (s + m)\zeta(s+m+1,a) = 1
\]
and, thus,
\[
\lim_{s\to -m} p_{m+1}(s)\zeta(s+m+1) = \frac{p_m(-m)}{m+2}
\]

This gives us the following equality for $m > 0$ an integer, by (\ref{absconv}), by our discussion above and by continuity.
\begin{align*}
-\frac{a^{m+1}}{m+1} = &\sum_{n=0}^m p_n(-m)\left( \zeta(-m+n,a)-\frac{1}{a^{-m+n}} \right) + \frac{p_{m}(-m)}{m+2}\\
 = &\frac{p_{m}(-m)}{m+2} + \sum_{n=0}^m p_n(-m) \left[\zeta(-m,a)-\frac{1}{a^{-m}} + \sum_{k=1}^\infty \frac{n^k}{k!} D_s^k \left[ \zeta(-m,a)-\frac{1}{a^{-m}}\right] \right]\\
= &\frac{p_{m}(-m)}{m+2} + \sum_{n=0}^m p_n(-m) \left[\zeta(-m,a)-\frac{1}{a^{-m}} + \sum_{k=1}^\infty \frac{n^k}{k!} \left[ \zeta^{(k)}(-m,a)-(\log(a))^k a^m\right] \right]\\
= &\frac{p_{m}(-m)}{m+2} +\left[\zeta(-m,a)-\frac{1}{a^{-m}} \right] \sum_{n=0}^m p_n(-m)  \\
&\phantom{weeeeeeeeeeee}+\sum_{k=1}^\infty \frac{1}{k!} \left[ \zeta^{(k)}(-m,a)-(\log(a))^k a^m\right]\sum_{n=0}^m n^kp_n(-m)
\end{align*}

The interchange between the finite sum and the series is justified by the absolute convergence of Taylor series within its radius of convergence. When we look at the zeta function, we get an interesting identity using the trivial zeros of the zeta function and the definition of $p_n(-m)$.


\begin{theorem}\label{-m}
For $m > 0$ an integer, we have that 
\begin{align*}
-\frac{1}{2m+1} = &\frac{1}{2m+1}- \sum_{n=0}^{2m} \frac{1}{(n+1)!}\prod_{j=0}^{n-1} (-2m+j) \\ &+\sum_{k=1}^\infty \frac{1}{k!}\zeta^{(k)}(-2m)\sum_{n=1}^{2m} \frac{n^k}{(n+1)!} \prod_{j=0}^{n-1} (-2m+j)
\end{align*}
and that, for $m \geq 0$,
\begin{align*}
-\frac{1}{2m+2} = &-\frac{1}{2m+2} - \sum_{n=0}^{2m+1} \frac{1}{(n+1)!}\prod_{j=0}^{n-1} (-2m+j)  \\ &+\sum_{k=0}^\infty \frac{1}{k!}\zeta^{(k)}(-2m-1)\sum_{n=1}^{2m+1} \frac{n^k}{(n+1)!} \prod_{j=0}^{n-1} (-2m-1+j)
\end{align*}
\end{theorem}
\qed

\subsection{Applying $G$ to Dirichlet $L$-functions}\label{secL}
Corollary \ref{cor} can be reframed in terms of the operator $G$ so that the equation corresponding to (\ref{HurDE}) does not {\it only} hold formally. Furthermore, Van Gorder's original result may also be extended to provide an operator equation involving Dirichlet $L$-functions.  

\begin{proposition}
Let $G$ be the operator defined above. Then, for a Dirichlet character $\chi \mod k$, and for $s \in \C\setminus \{1\}$, we have that,
\begin{align}
G\left[k^s L(s,\chi) - \sum_{r=1}^k \chi (r) \frac{k^s}{r^s} \right] = \frac{k^{s-1}}{s-1}\sum_{r = 1}^k  \frac{\chi(r)}{r^{s-1}}
\end{align}
\end{proposition}
\begin{proof}
Let $\chi$ be a Dirichlet character. From the definition of $G$,
\begin{align}\label{26}
    G\left[ k^s L(s,\chi) - \sum_{r=1}^k \chi (r) \frac{k^s}{r^s} \right] = \sum_{n=0}^\infty p_n (s) \left(k^s L(s+n,\chi) - \sum_{r=1}^k \chi (r) \frac{k^{s+n}}{r^{s+n}} \right)
\end{align}
But we know that, for $\chi$ a character mod $k$,
\begin{align}\label{Lid}
    L(s,\chi) = k^{-s}\sum_{r = 1}^k \chi(r) \zeta (s, r/k)
\end{align}
Thus, (\ref{26}) becomes,
\begin{align*}
    &\sum_{n=0}^\infty p_n (s) \left(k^s k^{-s}\sum_{r = 1}^k \chi(r) \zeta (s+n, r/k) - \sum_{r=1}^k \chi (r) \frac{k^{s+n}}{r^{s+n}} \right)\\
    &= \sum_{n=0}^\infty p_n (s) \left[\sum_{r = 1}^k \left( \chi(r) \zeta (s+n, r/k) - \chi (r) \frac{k^{s+n}}{r^{s+n}} \right)\right]\\
    &= \sum_{n=0}^\infty \sum_{r = 1}^k \left[ \chi(r) p_n (s) \left( \zeta (s+n, r/k) - \frac{k^{s+n}}{r^{s+n}} \right)\right]
\end{align*}
Now, since
$
\sum_{n=0}^\infty  p_n (s) \left( \zeta (s+n, r/k) - \frac{k^{s+n}}{r^{s+n}}\right)
$
converges absolutely by Lemma \ref{absconv}, we can change the order of summation to get,
\begin{align}\label{32}
    &\sum_{r = 1}^k \chi(r) \sum_{n=0}^\infty \left[ p_n (s) \left( \zeta (s+n, r/k) - \frac{k^{s+n}}{r^{s+n}} \right)\right]
\end{align}
But, by Corollary 2, we have that the last summation equals $\frac{k^{s-1}}{(s-1)r^{s-1}}$ for $r = 1,..,k$ and (\ref{32}) becomes
\begin{align*}
    \sum_{r = 1}^k \chi(r)\frac{k^{s-1}}{(s-1)r^{s-1}} = \frac{k^{s-1}}{s-1}\sum_{r = 1}^k  \frac{\chi(r)}{r^{s-1}}
\end{align*}
as desired.
\end{proof}

This gives us an identity of Dirichlet $L$-functions. Namely, for $\chi$ a character mod $k$, we have
$$L(s,\chi) - \sum_{r=1}^k  \frac{\chi (r)}{r^{s}} = \frac{1}{k(s-1)}\sum_{r = 1}^k \frac{\chi(r)}{r^{s-1}}  - \frac{1}{k^s} \sum_{n=1}^\infty \frac{\prod_{j=0}^{n-1}(s+j)}{(n+1)!} \left(k^{s+n} L(s+n,\chi) - \sum_{r=1}^k \chi (r) \frac{k^{s+n}}{r^{s+n}} \right)$$

\subsection{Relating $G^{-1}$ to the Hurwitz $\zeta$ and to Dirichlet $L$-functions}\label{inverse}

In \cite{VanGorder}, Van Gorder defines an inverse operator to $T$. We see in his proof of Theorem 4.1 that Van Gorder's definition of $T^{-1}$ was the inverse operator for $T$ he did not use the definition of $\exp(nD)$ but rather the characterization that it is a shift operator. Thus, we may use this construction to define an inverse operator to $G$, which we denote $G^{-1}$. Let $f$ be a complex valued function. Then $G^{-1}$ is given by
\begin{align}\label{inverse}
    G^{-1}[f](s) = \sum_{n=1}^\infty q_n(s)f(s+n)
\end{align}
Where $q_0(s) = 1$ and $q_n(s) = -\sum_{k=0}^{n-1}q_k(s)p_{n-k}(s+k)$. The formal proof that this operator is the inverse of $G$ is given by Van Gorder in the proof of Theorem 4.1 in \cite{VanGorder}. However, this argument does not give reference to where this inverse converges.  We will address this question in Section \ref{Ginvcon}, but first we will make clear the relationship between $G$ and the Bernoulli numbers which Van Gorder alluded to in \cite{VanGorder}. 

\section{Relationship to Bernoulli numbers}\label{Bern}
In his paper, Van Gorder alludes to a relationship between the $q_n(s)$ and Bernoulli numbers. We now provide a proof of such relationship.

\begin{proposition}\label{13}
For all integers $n \geq 0$ and all $s \in \C$, we have the identity
\begin{align*}
    q_n(s) = \frac{B_n}{n!}\prod_{j=0}^{n-1} (s+j)
\end{align*}
Where $B_n$ denotes the $n^{\text{th}}$ Bernoulli number.
\end{proposition}

\begin{proof}
We proceed with strong induction on $n$.

For $n=0$, $q_0(s) = 1 = \frac{B_0}{0!}$. 
Suppose that we proved our claim for all $0 \leq k \leq n-1$ for some $n-1 \geq 0$. We prove it for $n$. By the strong induction hypothesis and by the definition of $p_n(s)$, we have that, for all $s \in \C$

\begin{align*}
    q_{n}(s) &= -\sum_{k=0}^{n-1} q_k(s)p_{n-k}(s+k) \\
                &= -\sum_{k=0}^{n-1} \left[ \left( \frac{B_k}{k!}\prod_{j=0}^{k-1}(s+j) \right) \left( \frac{1}{(n+1-k)!}  \prod_{j=0}^{n-k-1} (s+k+j)\right) \right] \\
                &=  -\sum_{k=0}^{n-1} \left[ \left( \frac{B_k}{k!(n+1-k)!}\prod_{j=0}^{k-1}(s+j) \right)  \prod_{j=k}^{n-1} (s+j)\right] \\
                &=  -\sum_{k=0}^{n-1} \left[ \frac{B_k}{k!(n+1-k)!}\prod_{j=0}^{n-1}(s+j) \right]\\
                &=  \left(\prod_{j=0}^{n-1}(s+j) \right)\left(-\sum_{k=0}^{n-1} \frac{B_k}{k!(n+1-k)!} \right) \\
                &= \left(\prod_{j=0}^{n-1}(s+j) \right)\left(-\sum_{k=0}^{n-1} {n+1 \choose k}\frac{B_k}{(n+1)!} \right)
\end{align*}
Since the Bernoulli numbers have the recursive formula $B_0 = 1$ and $(n+1)B_n = -\sum_{k=0}^{n-1} {n+1 \choose k}B_k$ for $n > 0$, we conclude that
\begin{align*}
    q_n(s) &= \frac{1}{(n+1)!}\left(\prod_{j=0}^{n-1}(s+j) \right)\left(-\sum_{k=0}^{n-1} {n+1 \choose k}B_k \right) \\
    &=  \frac{(n+1)}{(n+1)!}B_n \prod_{j=0}^{n-1}(s+j) \\
    &= \frac{B_n}{n!}\prod_{j=0}^{n-1}(s+j)
\end{align*}
This completes our induction.
\end{proof}

Proposition \ref{13} gives a surprising connection between $G$ and Bernoulli numbers. We now proceed to use $G^{-1}$ to recover series representations of $\zeta(\cdot,a)$ and $L(\cdot,\chi)$.

\subsection{Using $G^{-1}$ to Represent the Hurwitz $\zeta$-function and Dirichlet $L$-functions}\label{Ginvcon}

Using the fact that $G^{-1}$ is an inverse operator to $G$, formally we have
\begin{align}\label{hurid}\zeta(s,a) - \frac{1}{a^s} = G^{-1} \left[ \frac{1}{(s-1)a^{s-1}}\right] =\sum_{n=0}^\infty \frac{q_n(s)}{(s+n-1)a^{s+n-1}} = \sum_{n=0}^\infty \left(\frac{B_n}{n!}\cdot \frac{\prod_{j=0}^{n-1}(s+j)}{(s+n-1)a^{s+n-1}} \right)\end{align}
and
\begin{align}\label{dlid}
k^s L(s,\chi) - \sum_{r=1}^k \chi (r) \frac{k^s}{r^s} = G^{-1}\left[\frac{k^{s-1}}{s-1}\sum_{r = 1}^k  \frac{\chi(r)}{r^{s-1}}\right] = \sum_{n=0}^\infty \frac{B_n\prod_{j=0}^{n-1}(s+j)}{n!}\left(\frac{k^{s+n-1}}{s+n-1}\sum_{r = 1}^k  \frac{\chi(r)}{r^{s+n-1}} \right) 
\end{align}
We now treat the convergence of these identities. Namely we will show that they only converge at negative integers.

\begin{proposition} The sums
$$\sum_{n=0}^\infty \frac{B_n}{n!}\cdot \frac{\prod_{j=0}^{n-1}(s+j)}{(s+n-1)a^{s+n-1}} 
\ \ \ \ \text{ and } \ \ \ \ 
\sum_{n=0}^\infty \frac{B_n\prod_{j=0}^{n-1}(s+j)}{n!}\left(\frac{k^{s+n-1}}{s+n-1}\sum_{r = 1}^k  \frac{\chi(r)}{r^{s+n-1}} \right) $$ diverge for all $s\in \mathbb{C}\setminus \mathbb{Z}_{\leq 0}$ and converge when $s= -M\in\Z_{\leq 0}$ giving 
$$G^{-1}[f](-M)=\sum_{n=1}^M (-1)^n B_n \frac{M!}{n! (M-n)!} f(-M+n).$$ \end{proposition}

\begin{proof}
By Euler's formula, for each $k\in \Z_{>0}$ 
$$\zeta(2k) =(-1)^{k+1}\frac{(2\pi)^{2k}B_{2k}}{2(2k)!}$$ 
where $B_{2k}$ is the $2k$-th Bernoulli number. Since $\lim_{k\to\infty}\zeta(k)  = 1$, 
$$\left|\frac{(2\pi)^{2k}B_{2k}}{2(2k)!}
\right|\sim 1$$
and so 
$$\left|\frac{B_{2k}}{(2k)!}
\right|\sim \frac{2}{(2\pi)^{2k}}.$$

Now note that 
\begin{align*} \sum_{n=0}^\infty \frac{B_n}{n!}\cdot \frac{\prod_{j=0}^{n-1}(s+j)}{(s+n-1)a^{s+n-1}} 
&=\sum_{n=0}^\infty \frac{B_{2n}}{(2n)!}\cdot \frac{\prod_{j=0}^{2n-1}(s+j)}{(s+2n-1)a^{s+2n-1}} \\
& =\sum_{n=0}^\infty \frac{B_{2n}}{(2n)!}\cdot \frac{\prod_{j=0}^{2n-2}(s+j)}{a^{s+2n-1}} =: \sum_{n=0}^\infty c_n
\end{align*}
since the $2n+1$ Bernoulli numbers are $0$. 
Now 
\begin{align*}
\left|\frac{c_{n+1}}{c_n}\right|
&=\left| \frac{B_{2n+2}\cdot \prod_{j=0}^{2n}(s+j)}{a^{s+2n+1}(2n+2)!}\cdot  \frac{a^{s+2n-1}(2n)!}{B_{2n}\cdot \prod_{j=0}^{2n-2}(s+j)}\right|\\
& \sim\left| \frac{2\cdot \prod_{j=0}^{2n}(s+j)}{a^{s+2n+1}(2\pi)^{2n+2}}\cdot  \frac{a^{s+2n-1}(2\pi)^{2n}}{2\cdot \prod_{j=0}^{2n-2}(s+j)}\right|\\
&=\left| \frac{ (s+2n)(s+2n-1)}{(2\pi a)^{2}}\right|\to \infty
\end{align*}
as $n\to \infty$.

Similarly \begin{align*} \sum_{n=0}^\infty \frac{B_n\prod_{j=0}^{n-1}(s+j)}{n!}\left(\frac{k^{s+n-1}}{s+n-1}\sum_{r = 1}^k  \frac{\chi(r)}{r^{s+n-1}} \right)
&=\sum_{n=0}^\infty \frac{B_{2n}}{(2n)!}\cdot \prod_{j=0}^{2n-2}(s+j)\cdot k^{s+2n-1}\sum_{r = 1}^k  \frac{\chi(r)}{r^{s+2n-1}} \\
&=: \sum_{n=0}^\infty d_n
\end{align*}
and
\begin{align*}
\left|\frac{d_{n+1}}{d_n}\right|
&=\left| \frac{B_{2n+2}\cdot \prod_{j=0}^{2n}(s+j)k^{s+2n+1} }{(2n+2)!}\sum_{r = 1}^k  \frac{\chi(r)}{r^{s+2n+1}} \cdot  \frac{(2n)!}{B_{2n}\cdot \prod_{j=0}^{2n-2}(s+j)\cdot k^{s+2n-1}\sum_{r = 1}^k  \frac{\chi(r)}{r^{s+2n-1}}}\right|\\
& \sim\left| \frac{2\cdot \prod_{j=0}^{2n}(s+j)k^{2}}{(2\pi)^{2n+2}}\cdot  \frac{(2\pi)^{2n}}{2\cdot \prod_{j=0}^{2n-2}(s+j)}\sum_{r = 1}^k  \frac{\chi(r)}{r^{s+2n+1}}\left(\sum_{r = 1}^k  \frac{\chi(r)}{r^{s+2n-1}}\right)^{-1} \right|\\
&\sim\left| \frac{ (s+2n)(s+2n-1)}{(2\pi )^{2}} \right|\to \infty
\end{align*} and $n\to \infty$

To see that these sums converge for $s\in \Z_{\leq 0}$, let $s = -M$ for some $M\in \Z_{\geq 0}$, and notice that for all $n> M$,
\[ q_n(s) = q_n(-M) = \frac{B_n}{n!} \prod_{j=0}^{n-1} (-M+j) = 0  \]
since there will be a $-M+M$ term in the product. It follows that
\begin{align*}
G^{-1}[f](-M) &= \sum_{n=1}^M q_n(-M) f(-M+n) \\
&= \sum_{n=1}^M \frac{B_n}{n!}(-M)\cdot (-M+1)\cdot ...\cdot (-M+n-1) \cdot f(-M+n)\\
&=  \sum_{n=1}^M (-1)^n B_n \frac{M!}{n! (M-n)!} f(-M+n) \,.
\end{align*}
\end{proof}

\subsection{Euler-Maclaurin Formula and Analytic Continuation of Dirichlet $L$-functions}\label{sec:EM}

Notice that, for the Hurwitz zeta function, the coefficients in the sum are the same as the coefficients in the Euler-Maclaurin summation formula (which rarely a convergent series for all $s \in \C$). 
For Dirichlet $L$-functions, however, there is no clear way to apply the Euler-Maclaurin summation formula. One can, however, derive a series representation of Dirichlet $L$-functions from the Euler-Maclaurin summation formula for the Hurwitz zeta using the identity in (\ref{Lid}), to get for $\chi$ a character mod $k$,
\begin{align}\label{emid}
    L(s,\chi) = k^{-s}\sum_{r=1}^k \chi(r)\left[  \frac{k^s}{r^s}+ \sum_{n=0}^\infty \left(\frac{B_n}{n!}\cdot \frac{k^{s+n-1}\prod_{j=0}^{n-1}(s+j)}{r^{s+n-1}(s+n-1)} \right) \right]
\end{align}
(which is a rearrangement of (\ref{dlid})). 
The proof we give does not follow as given since it relied on he convergence of the series above. However, we can still give a proof of the meromorphic continuations of $L(s,\chi)$ using Euler-Maclaurin series.  

The Euler-Maclaurin formula says that if $f$ and all of its derivatives go to zero as $x\to \infty$, 
$$\sum_{n= a}^\infty f(n) =\int_a^\infty f(x)\, dx +\frac{1}{2}f(a) -\sum_{\ell = 2}^k\frac{(-1)^\ell}{\ell !}
f^{(\ell -1)}(a)B_\ell - \frac{(-1)^k}{k!}\int_a^{\infty}f^{(k)}(x)\psi_k(x)\,dx$$
where $\psi_k(x) =B_k(\{x\})$ for $B_k$ the Bernouli polynomials and  $\{x\}$ the fractional part of $x$. Note that the {\it Bernouli polynomials} $B_k(x)$ are defined by the following three properties: 
\begin{enumerate}
    \item $B_0(x)=1$
    \item $B_k'(x) = kB_{k-1}(x)$ for $x=1,2,\dots$ 
    \item $\int_0^1 B_k(x)\, dx=0$  for $x=1, 2, \dots$
\end{enumerate}

For $\text{Re}(s)>1$,
$$\int_0^\infty (x+a)^{-s}\,dx 
=\frac{(x+a)^{-s+1}}{-s+1}\Big|_{x=0}^\infty = \frac{a^{-s+1}}{s-1}
$$
For the Hurwitz zeta function, for $\text{Re}(s)>1$, the Euler-Maclaurin formula gives 
\begin{align*}
    \zeta(s,a)& = \int_0^\infty (x+a)^{-s}\,dx +\frac{1}{2a^s}
    -\sum_{\ell =2 }^k \frac{(-1)^\ell}{\ell !}
f^{(\ell -1)}(a)B_\ell - \frac{(-1)^k}{k!}\int_0^{\infty}f^{(k)}(x)\psi_k(x)\,dx\\
&=  \frac{a^{-s+1}}{s-1} +\frac{1}{2a^s}
    -\sum_{\ell =2 }^k \frac{(-1)^\ell}{\ell !}
(-1)^{\ell -1}\prod_{j=1}^{\ell-2}(s+j) a^{-s-\ell +1}B_\ell \\
& \ \ \ \ \ \ \ \ \ \ \ \ \ \ \ - \frac{(-1)^k}{k!}\int_0^{\infty}(-1)^k \prod_{j=0}^{k-1}(s+j)(x+a)^{-s-k}\psi_k(x)\,dx\\
&=  \frac{a^{-s+1}}{s-1} +\frac{1}{2a^s}
    +\sum_{\ell =2 }^k \frac{1}{\ell !}
\prod_{j=1}^{\ell-2}(s+j) a^{-s-\ell +1}B_\ell - \frac{1}{k!}\int_0^{\infty} \prod_{j=0}^{k-1}(s+j)(x+a)^{-s-k}\psi_k(x)\,dx
\end{align*}

 Recall that when $\chi$ is a Dirichlet character modulo $q$, $L(s,\chi)= q^{-s} \sum_{r=1}^q\chi(r)\zeta(s,r/q)$.  Thus for $\text{Re}(s)>1$, 
\begin{align}\label{eq:LEM} L(s,\chi)&= q^{-s} \sum_{r=1}^q\chi(r)\Big(\frac{(r/q)^{-s+1}}{s-1} +\frac{1}{2(r/q)^s}
    +\sum_{\ell =2 }^k \frac{1}{\ell !}
\prod_{j=1}^{\ell-2}(s+j) (r/q)^{-s-\ell +1}B_\ell\\
& \ \ \ \ \ \ \ \ \ \ \  \ \ \  \ \ \  \ \ \  \ \ \  \ \ \  - \frac{1}{k!}\int_0^{\infty} \prod_{j=0}^{k-1}(s+j)(x+(r/q))^{-s-k}\psi_k(x)\,dx\Big)\nonumber
 \end{align}
 
 \begin{proposition} For $\chi$ a Dirichlet character modulo $q$,  $L(s,\chi)$ defined by (\ref{eq:LEM}) for $\text{Re}(s)>1$ can be analytically continued to $\C-\{1\}$ where it has a simple pole at $s=1$.\end{proposition}
 
 \begin{proof}
 For $k=1$, (\ref{eq:LEM}) becomes  
 \begin{align*} L(s,\chi)&= q^{-s} \sum_{r=1}^q\chi(r)\Big(\frac{(r/q)^{-s+1}}{s-1} +\frac{1}{2(r/q)^s}
     - s\int_0^{\infty} (x+(r/q))^{-s-1}\psi_1(x)\,dx\Big)
 \end{align*}
 when $\text{Re}(s)>1$.  Rearranging, we get 
  \begin{align}\label{eq:mero} L(s,\chi)-q^{-s} \sum_{r=1}^q\chi(r)\frac{(r/q)^{-s+1}}{s-1} 
  &= q^{-s} \sum_{r=1}^q\chi(r)\Big(\frac{1}{2(r/q)^s}
     - s\int_0^{\infty} (x+(r/q))^{-s-1}\psi_1(x)\,dx\Big)
 \end{align}
 Since $|\psi_1(x)|\leq 1/2$,
 $$\Big|\int_0^{\infty} (x+(r/q))^{-s-1}\psi_1(x)\,dx\Big|
 \leq \int_0^{\infty} (x+(r/q))^{-s-1}\,dx$$ and this last integral is finite when $\text{Re}(s)>0$. Using the right side of (\ref{eq:mero}) to define the left side of (\ref{eq:mero}) for $0<\text{Re}(s)\leq 1$.  We see that $L(s,\chi)$ has a simple pole at $s=1$ and is analytic for all other points of $\text{Re}(s)>0$. 
 
To extend this $\text{Re}(s)<0$, note that $\psi_k$ is bounded and periodic on $\R$ with period 1 and equal to $B_k(x)$ on $[0,1)$. Thus the integral on right side  of  (\ref{eq:LEM}) is convergent for all $\text{Re}(s)+k>1$ and holomorphic for $\text{Re}(s)>1-k$. By repeating the argument above we have analytically continued $L(s,\chi)$ for $\text{Re}(s)>1-k$ for all $k=1,2,3,\dots$.
 \end{proof}

\section{Approximations}\label{approx}

In Theorem \ref{main}, we establish that $T\left[ \zeta (s, a) - \frac{1}{a^s} \right]=\sum_{n=0}^{\infty} p_n(s)\exp(nD)\left[\zeta(s,a)- \frac{1}{a^s}\right]$ diverges for $s\in \C $ so clearly it is not the case that 
$T\left[\zeta (s,a) - \frac{1}{a^s}\right] $ converges pointwise to $\frac{1}{(s-1)a^{s-1}}$ for $s\in \C $.  However, it may be the case that truncating $T$ may provide a good approximation even at values where the series does not converge.  
Of course, when considering the operator $T=\sum_{n=0}^\infty p_n(s) [id + \sum_{k=1}^\infty \frac{n^k}{k!} D_s^k]$, there are two sums that we may consider truncating:  the sum over $n$ and the sum over $k$.  In what follows we will truncate in $n$.


Consider 
$$T_N(s,a):=\sum_{n=0}^{N} p_n(s)\exp(nD)\left[\zeta(s,a)- \frac{1}{a^s}\right]$$ Note that though $T$ does not converge when applied to the zeta function, $T_N$ may converge when applied to $\zeta(s,a)$.
Recall that from Proposition \ref{taylordiv}, for each $s$ there is some $N'$ so that $\exp(N'\cdot D) \left[\zeta(s,a)-\frac{1}{a^s} \right]$ diverges.  However, from the Taylor series expansion,
$$    \zeta(s+n,a) - \frac{1}{a^{s+n}} = \sum_{k = 0}^\infty \frac{D_s^k \left( \zeta(s,a) - \frac{1}{a^s}\right)}{k!}n^k $$
  converges pointwise for $|(s+n)-s|=|n|<|s-1|$ since $\zeta(z,a)$ has a pole at $z=1$.
  Thus we have 
\begin{align*}T_N(s,a):=\sum_{n=0}^{N} p_n(s)\exp(nD)\left[\zeta(s,a)- \frac{1}{a^s}\right] =\sum_{n=0}^{N} p_n(s)\left[\zeta(s+n,a)- \frac{1}{a^{s+n}}\right]=:G_N(s,a)\end{align*} for $N<|s-1|$. In other words, in the region of convergence (away from $s=1$), the truncation of $T$ in $n$ is equal to the truncation of the shift operator $G$.

In what follows we will use complex plots to examine whether $T_N(s,a)$ is good approximation to $\frac{1}{(s-1)a^{s-1}}$ for $s$ in the region of convergence for $T_N$.
 Figures \ref{fig:approx1} and  \ref{fig:approx2} are obtained using ``complex$\_$plot" in Sage.  This function takes a complex function of one variable, $f(z)$ and plots output of the function over the specified x$\_$range and y$\_$range. The magnitude of the output is indicated by the brightness (with zero being black and infinity being white) while the argument is represented by the hue.  The hue of red is positive real, and increasing through orange, yellow, as the argument increases and the hue of green is positive imaginary. Note that, for simplicity, both figures only plot the specific case of the Riemann zeta function (when $a=1$).

Figure \ref{fig:approx1} (a) is the complex plot of $\frac{1}{s-1}$ the right side of Van Gorder's equation (\ref{zetaDE}). Subfigures (b), (c), (d) and (e) of approximations $G_N(s,1)$ of the left side of Van Gorder's equation (\ref{zetaDE}) for $N=1, 10, 50$ and $100$. Note that the domain of the plots in subfigures (d) and (e) has been expanded.  




%

\begin{figure}
\begin{center}
\begin{subfigure}{0.49\textwidth}
  \includegraphics[width=\textwidth]{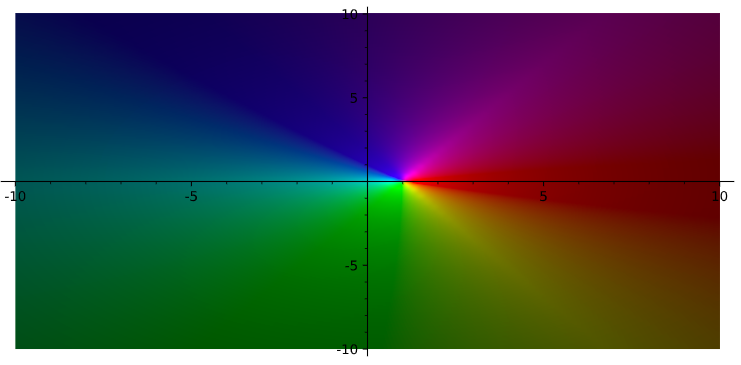}
    \caption{$\frac{1}{s-1}$  }
  \end{subfigure}
  
\begin{subfigure}{0.49\textwidth}
  \includegraphics[width=\textwidth]{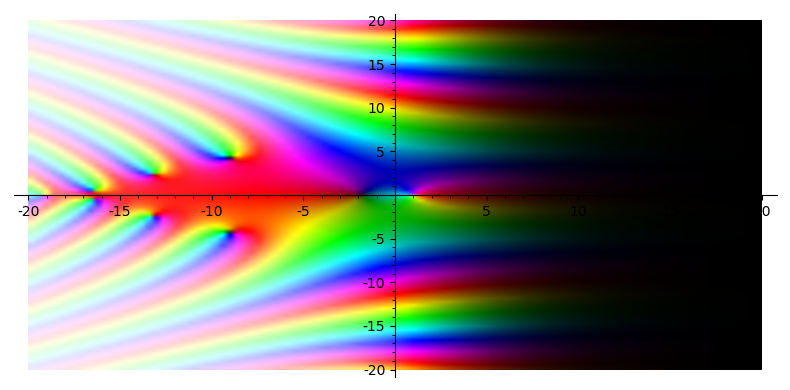}
  \caption{ $ G_1(s,1)$}
  \end{subfigure}
 \begin{subfigure}{0.49\textwidth}
  \includegraphics[width=\textwidth]{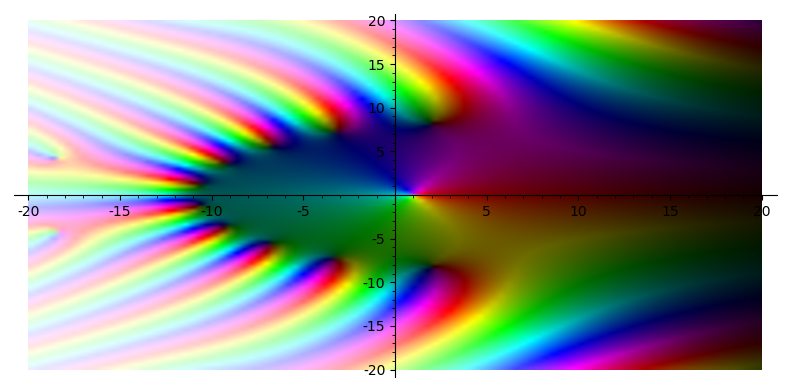}
    \caption{$G_{10}(s,1)$}
    \end{subfigure}

        \begin{subfigure}{0.49\textwidth}
  \includegraphics[width=\textwidth]{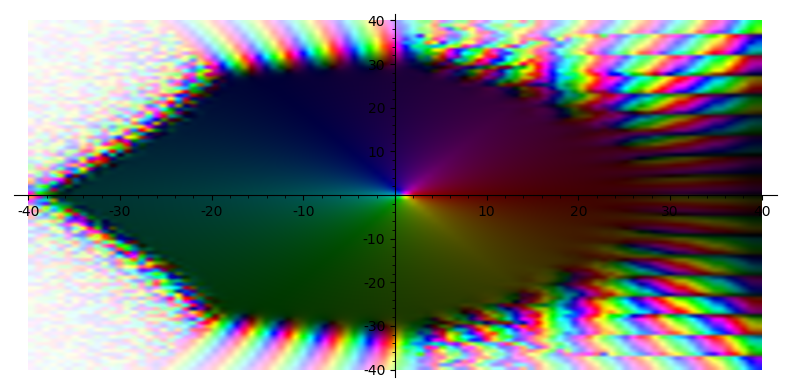}
    \caption{$G_{50}(s,1)$}
    \end{subfigure}
          \begin{subfigure}{0.49\textwidth}
  \includegraphics[width=\textwidth]{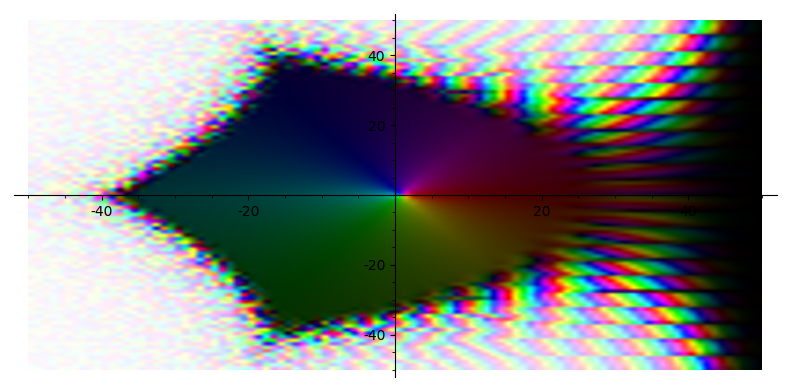}
    \caption{$G_{100}(s,1)$}
        \end{subfigure}

\end{center}
  \caption{Complex plots of $\frac{1}{s-1}$ and $G_N(s,1)$ for $N=1, 10, 50,$ and $100$. Note that the domains for (a), (b) and (c) are $(-20,20)$ on both the real and imaginary axes and for (d) and (e) is $(-40,40)$.}\label{fig:approx1}
\end{figure}

We see in Figure \ref{fig:approx1} 
that the first term of the expansion $G_N(s,1)$ is a good approximation of $\frac{1}{s-1}$ near the singularity $s=1$. (Note: This is not surprising since as $N\to \infty$, we have shown $G_N(s,1)\to G[\zeta(s)-1] = \frac{1}{s-1}$.) We also see from Figure \ref{fig:approx1} that as $N$ grows, the region on which $G_N(s,1)$ is a good approximation of $\frac{1}{s-1}$ also expands.  

As we observed at the beginning of this section, $T_N=G_N$ away from $s=1$ so $T_N(s,a)$ may be a good approximation to $\frac{1}{(s-1)a^{s-1}}$ for $s$ far enough from 1 (i.e. outside the circle $|s-1|=N$).  One might wonder whether the expanding region on which $G_N(s,1)$ is a good approximation of $\frac{1}{s-1}$ will break into the region of convergence of $T_N(s,1).$

Figure \ref{fig:approx2} contains complex plots of $G_N(s,1)-\frac{1}{s-1}$ for $N=10, 50$ and $100$ along with the circle $|s-1|=N$ for $N=10, 50$ and $100$.  The inclusion of this circle in plot is to be able to identify where where $G_N(s,1)= T_N(s,1)$ (outside $|s-1|=N$).

\begin{figure}
\begin{center}

     \begin{subfigure}{0.31\textwidth}
  \includegraphics[width=\textwidth]{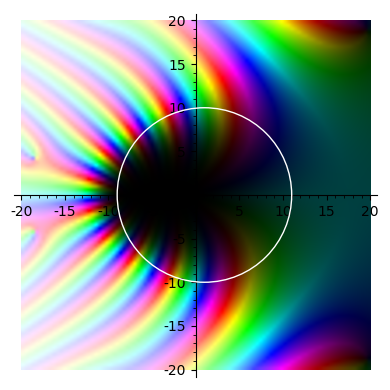}
    \caption{$N=10$ }
    \end{subfigure}     
     \begin{subfigure}{0.31\textwidth}
  \includegraphics[width=\textwidth]{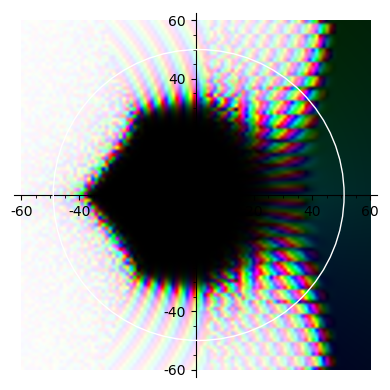}
    \caption{$N=50$}
    \end{subfigure}
        \begin{subfigure}{0.31\textwidth}
  \includegraphics[width=\textwidth]{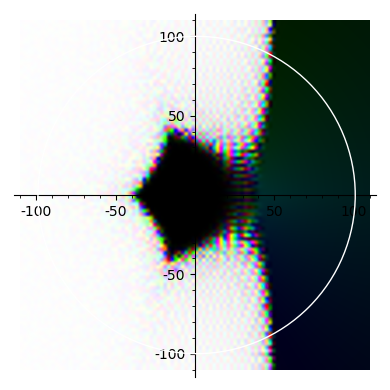}
    \caption{$N=100$}
    \end{subfigure}

\end{center}
    \caption{Complex plots of $G_N(s,1)-\frac{1}{s-1}$ for $N=10, 50,$ and $100$ as well as the circle $|s-1|=N$ for $N=10, 50,$ and $100$  (in white) which shows the region for which for $T_N$ converges and equal $G_N$.}\label{fig:approx2}

  \end{figure}

Examining Figure \ref{fig:approx2}, we see that this ``region of good approximation" is not expanding as quickly as the region of convergence for $T_N$ (the radius of $|s-1|=N$). Thus, as $N$ grows, $T_N(s,1)$ actually seems to be a less reasonable approximation for $\frac{1}{s-1}$.

Many natural questions remain about both the accuracy of the approximation $T_N$ in this region and the accuracy of the approximation $G_N$ in the disk and the rate of convergence of $G$.

\section{Acknowledgements}
We would like to thank Dr.\,Paul Young for calling our attention to our mistake with what was previously Section 4.1 and for his helpful comments. K. K-L. acknowledges support from NSF Grant number DMS-2001909.

\section{Conclusion}

Using the operator $G$ as opposed to $T$ allows us to provide more than formal justification for the differential equation (\ref{zetaDE}) as well as the corresponding generalizations to the Hurwitz zeta function and Dirichlet $L$-function. However, it is important to note that $G$ is not a differential operator and so, in fact, this does not provide support for their being a non-algebraic differential equation which zeta satisfies. 




\end{document}